\documentclass[a4paper,twoside,final,reqno]{amsart}
\usepackage[utf8]{inputenc}
\usepackage[T1]{fontenc}
\usepackage[english]{babel}
\usepackage{etex,fixltx2e}
\usepackage{braket}
\usepackage{lmodern}
\usepackage{amsmath,amssymb,ifthen}
\usepackage{mathtools}

\newcommand{\bbA}{{\mathbb A}}
\newcommand{\bbC}{{\mathbb C}}
\newcommand{\bbN}{{\mathbb N}}
\newcommand{\bbQ}{{\mathbb Q}}
\newcommand{\bbR}{{\mathbb R}}
\newcommand{\bbZ}{{\mathbb Z}}
\newcommand{\calO}{{\mathcal O}}
\newcommand{\frakd}{{\mathfrak d}}
\newcommand{\frakm}{{\mathfrak m}}
\newcommand{\frakp}{{\mathfrak p}}
\newcommand{\frakI}{{\mathfrak I}}
\newcommand{\frakM}{{\mathfrak M}}
\newcommand{\frakR}{{\mathfrak R}}

\newcommand{\Ima}{{\ensuremath{\operatorname{\frakI}}}}
\newcommand{\Rea}{{\ensuremath{\operatorname{\frakR}}}}
\newcommand{\Tr}{{\ensuremath{\operatorname{Tr}}}}
\newcommand{\isom}{\mathrel{\cong}}
\newcommand{\GL}{{\ensuremath{\operatorname{GL}}}}
\DeclarePairedDelimiter\abs{\lvert}{\rvert}

\newcommand{\im}{{\mathrm i}}
\newcommand{\tsum}{{\textstyle\sum}}

\newcommand{\compmod}{\prescript{\star}{}{\calO}}

\theoremstyle{plain}
\newtheorem{thm}{Theorem}[section]
\newtheorem{lem}[thm]{Lemma}
\newtheorem{cor}[thm]{Corollary}
\theoremstyle{definition}
\newtheorem{rem}[thm]{Remark}
\newtheorem{defn}[thm]{Definition}
\newtheorem{exmp}[thm]{Example}

\begin{document}

\title{Adelic Geometry and Polarity}
\author{Carsten Thiel}
\address{Fakult\"at f\"ur Mathematik, Otto-von-Guericke
Universit\"at Mag\-deburg, Universit\"atsplatz 2, D-39106 Magdeburg}
\email{carsten.thiel@ovgu.de}



\begin{abstract}
In the present paper we generalise transference theorems 
from the classical geometry of numbers to
the geometry of numbers over the ring of adeles of a number field.
To this end we introduce a notion of polarity for adelic convex bodies.
\end{abstract}

\maketitle

\section{Introduction}

By a convex body $S$ in the $m$-dimensional Euclidean space $\bbR^m$, we mean a compact and convex set $S\subset\bbR^m$
with non-empty interior, which we assume to be $0$-symmetric, i.e.\ $S=-S$.
A lattice $\Lambda$ is a discrete $\bbZ$-submodule of $\bbR^m$ of full rank.

Given a convex body $S$ and a lattice $\Lambda$ in $\bbR^m$,
the $i$-th successive minimum $\lambda_i(S,\Lambda)$ for $1\leq i\leq m$
of $S$ with respect to $\Lambda$ is defined as
\[\lambda_i(S,\Lambda)\coloneqq\inf\Set{\lambda>0 | \lambda S\cap\Lambda \text{ contains at least }i\text{ linearly independent elements}}\,.\]

With a convex body $S$ and a lattice $\Lambda$ we can associate the polar body
\[S^\star\coloneqq\Set{x\in\bbR^m | \left<x,y\right> \leq 1\ \forall\,y\in S}\]
and the polar lattice
\[\Lambda^\star\coloneqq\Set{x\in\bbR^m | \left<x,y\right> \in\bbZ\ \forall\,y\in \Lambda}\,,\]
where $\left<\,\cdot\,,\,\cdot\,\right>$ denotes the standard scalar product on $\bbR^m$.
We have $(\bbZ^m)^\star=\bbZ^m$ and $B_m^\star=B_m^{\phantom{\star}}$ for the Euclidean unit ball.

A classical inequality, first investigated by Mahler, is the transference result
\begin{equation}\label{eq:classicalinequality}
1\leq \lambda_i(S,\Lambda) \lambda_{m-i+1}(S^\star,\Lambda^\star) \leq m^{3/2}\,,
\end{equation}
for $1\leq i\leq m$. For the lower bound see Gruber \cite[\textsection\,5]{MR1242995},
while the upper bound follows from Banaszczyk \cite[Thm.~2.1]{banaszcyknewbounds}.

\medskip

We provide a generalisation of this inequality and of the notion of polarity
to the geometry of numbers 
over the ring of adeles of an algebraic number field.

The theory of adelic geometry of numbers arises in the context of Siegel's Lemma,
which asks for a small integral solution to a system of linear equations with integer coefficients.
Answers by Thue, Siegel and others usually involve counting arguments or Minkowski's theorems on successive minima, cf \cite{schmidt.lnm1467}.
In order to allow coefficients and solutions from an algebraic number field,
Bombieri and Vaaler in \cite{bombierivaalersiegelslemma} proved
an adelic variant of Minkowski's second theorem on successive minima.
A comprehensive overview of adelic geometry of numbers can be found in \cite{thunderremarksonadelic}.

The theory has been further generalised, as has Siegel's Lemma, with recent results by 
Fukshansky \cite{fukshanskysiegelslemma}, \cite{fukshanskyalgebraicpoints} and Gaudron \cite{gaudron}
on the number of algebraic points in bounded regions.
Further work on Siegel's Lemma for the algebraic closure of $\bbQ$
by Roy and Thunder~\cite{roythunderabsolutesiegel} involves the study of twisted heights.
Using these heights they introduce a different notion of adelic polarity 
and an analogous statement of (\ref{eq:classicalinequality}) in terms of these heights, 
which has most recently been extended by Rothlisberger~\cite{rothlisberger}.

The present paper however uses a more geometric approach, 
directly extending the classical notion of polarity to the adelic setting.

To this end we fix an algebraic number field $K$ of degree $d$ over $\bbQ$,
with field discriminant $\Delta_K$, cf.~\cite[Kap.\,I]{neukirch}.
We will use geometry of numbers over the ring of adeles $K_\bbA$ of $K$ of rank $n\in\bbN$.

The definitions of an adelic convex body $S$
and the adelic successive minima $\lambda_i(S)$ for $1\leq i,j \leq n$, 
will be provided in Section~\ref{sec:adelicgeometryofnumbers}.
For our definition of polar adelic body see Definition~\ref{def:adelicpolarbody}.

The main results of this paper are the following.

\begin{thm}\label{thm:adelicpolarupper}
Let $S$ be an adelic convex body, $S^\star$ its polar and let $\lambda_i(S)$, 
$\lambda_j(S^\star)$ ($1\leq i,j\leq n$) be the successive minima of $S$ and $S^\star$ respectively.
Then for $1\leq\ell\leq n$
\[\lambda_\ell(S)\lambda_{n-\ell+1}(S^\star)\leq (nd)^{3/2} \,.\]
\end{thm}

In view of the classical result (\ref{eq:classicalinequality}) we are also interested in a lower bound,
which however we can not proof in full generality.
For a special class of adelic convex body and for $K$ totally real or a CM-field
(i.e.\ a field of complex multiplication) we get the following estimate.

\begin{thm}\label{thm:adelicpolarlower}
Let $K$ be totally real or a CM-field and
let $S$ be an adelic convex body, with the additional requirement that 
for all complex places $v$, we have $S_v=\alpha S_v$ for $\alpha\in\bbC$ with $\abs{\alpha}=1$.
Let $S^\star$ be its polar and let $\lambda_i(S)$, 
$\lambda_j(S^\star)$ ($1\leq i,j\leq n$) be the successive minima of $S$ and $S^\star$ respectively,

Then for $1\leq\ell\leq n$
\[\tfrac{1}{\sqrt[d]{\abs{\Delta_K}}}\leq \lambda_\ell(S)\lambda_{n-\ell+1}(S^\star)\,.\]
\end{thm}

Notice, that in the case $K=\bbQ$ these results reduce to the classical statement (\ref{eq:classicalinequality}).
Finally, Example~\ref{exmp:bbQsqrt2} shows that the lower bound is sharp, at least for $n=1$.

\section{Adelic Geometry of Numbers}\label{sec:adelicgeometryofnumbers}

We start by giving a brief overview of the ring of adeles of an
algebraic number field $K$ of degree $d$ over $\bbQ$. 
For more details and proofs we refer to \cite[Ch.~IV]{weilbasicnumber} and \cite[Ch.~VI]{knappadvancedalgebra}.
Let $r$ be the number of real and $s$ the number of pairs of complex embeddings 
of $K$ into $\bbC$. Then $d=r+2s$. Denote by 
$\calO$ the ring of algebraic integers of $K$ and by $\Delta_K$ its field discriminant.

Let $M(K)$ be its set of places.
For $v\in M(K)$ we write $v\nmid\infty$ for non-archimedian places
and  $v\mid\infty$ for the archimedian ones.
For the corresponding absolute value on $K$ we write $\abs{\,\cdot\,}_v$.
We normalize it to extend either the usual absolute value on $\bbQ$ for archimedian places 
or the $p$-adic absolute value for a prime $p$.
Then the local field $K_v$ is the completion of $K$ with respect to $v$.
For $v\nmid\infty$ let $\calO_v$ be the local ring of integers.

Let $K_\bbA$ be the ring of adeles of $K$ and $K_\bbA^n$ the standard module of rank $n\geq 2$,
i.e.\ the $n$-fold product of adeles.
Recall that $K_\bbA$ is the restricted direct product of the $K_v$ with respect to the $\calO_v$.
For any $v\in M(K)$ let $d_v=[K_v:\bbQ_v]$ be the local degree ($\bbQ_\infty\isom\bbR$).
Then for all primes $p\in\bbZ$
\begin{equation}\label{eq:productformulaetc}
d=\sum_{v\mid p} d_v\quad \text{and}\quad d=\sum_{v\mid \infty} d_v\,,\quad\text{and also}\quad
\prod_{v\in M(K)} \abs{a}_v^{d_v}=1
\end{equation}
for all non-zero $a\in K$.

Denote by $\sigma_i$, $1\leq i \leq r$ the embeddings of $K$ into $\bbR$
and by $\sigma_{r+i}=\overline{\sigma}_{r+i+s}$, $1\leq i\leq s$ the pairs 
of embeddings of $K$ into $\bbC$, so $d=r+2s$.
We call $K$ \emph{totally real}, if $s=0$,
and we call $K$ a \emph{CM-field}, if it is a quadratic extension of a totally real field with $r=0$.
Then
\begin{align*}
\iota &\colon x\mapsto\bigl(\sigma_1(x),\ldots,\sigma_r(x),
\sigma_{r+1}(x),\ldots,\sigma_{r+s}(x)\bigr)\\
\shortintertext{and}
\overline{\iota} &\colon x\mapsto\bigl(\sigma_1(x),\ldots,\sigma_r(x),
\overline{\sigma}_{r+1}(x),\ldots,\overline{\sigma}_{r+s}(x)\bigr)
\end{align*}
are embeddings of $K$ into $K_\infty\coloneqq\prod_{v\mid\infty}K_v$.

There is a canonical isomorphism $\rho\colon K_\infty\rightarrow \bbR^{d}$ with
\[\begin{multlined}
\rho\bigl(x_1,\ldots,x_r,x_{r+1},\ldots,x_{r+s}\bigr)= \\
\qquad\bigl(x_1,\ldots,x_r,\Rea(x_{r+1}),\Ima(x_{r+1}),\ldots,\Rea(x_{r+s}),\Ima(x_{r+s})\bigr)\,.
\end{multlined}\]
Here $\Rea$ and $\Ima$ denote real and imaginary parts respectively.

Together we get $(\rho\circ\iota)\colon K\hookrightarrow \bbR^d$,
\[x\mapsto\bigl(\sigma_1(x),\ldots,\sigma_r(x),
\Rea(\sigma_{r+1}(x)),\Ima(\sigma_{r+1}(x)),\ldots,\Rea(\sigma_{r+s}(x)),\Ima(\sigma_{r+s}(x))\bigr)\,.\]

In the rank-$n$-case let $K_\infty^n\coloneqq\prod_{v\mid\infty}K_v^n$,
\[
\iota^n \coloneqq (\sigma_1^n,\ldots,\sigma_r^n,\sigma_{r+1}^n,\ldots,\sigma_{r+s}^n) 
\colon K^n\rightarrow K_\infty^n\,,
\quad\text{$\overline{\iota}^n$ respectively,}\]
where the $\sigma_i$ act componentwise.
Similarily $\rho^n\colon K_\infty^n\rightarrow\bbR^{nd}$.

\begin{defn}\label{def:adelicconvexbody}
For each $v\nmid\infty$ let $S_v$ be a free $\calO_v$-module of full rank,
where $S_v=\calO_v^n$ for all but finitely many $v$.
In other words, for any $v\nmid\infty$ there is an $A_v\in\GL_n(K_v)$ such that
$S_v=A_v^{-1}\calO_v^n$, where $A_v$ is the identity for all but finitely many $v$.
For $v\mid\infty$ we have $K_v\isom\bbR$ or $K_v\isom\bbC$. 
In this case let $S_v$ be a $0$-symmetric compact convex body with 
non-empty interior in $\bbR^n$ or $\bbC^n\isom\bbR^{2n}$ respectively.
Then the set
\[S = \prod_{v\nmid\infty} S_v \times \prod_{v\mid\infty} S_v\]
is called a closed symmetric \emph{adelic convex body}. If necessary, we denote $S_\infty=\prod_{v\mid\infty} S_v$.
\end{defn}

For $(x_v)_v\in K_\bbA^n$ we define the scalar multiple $(y_v)_v=\lambda(x_v)_v$ for $\lambda\in\bbR^+$ by
\[
y_v\coloneqq\begin{cases}\phantom{\lambda}x_v &\text{if } v\nmid\infty\,,\\
\lambda x_v &\text{if } v\mid\infty\,.\end{cases}
\]

\begin{defn}\label{def:adelsuccmindilat}
The \emph{$i$-th successive minimum} of the adelic convex body $S$ is
\[\lambda_i(S)=\inf\{\lambda>0 \mid \exists\, x_1,\ldots,x_i\in K^n\text{ lin.\ indep.\ over }K
\text{ s.t. } x_j\in\lambda S \text{ for all }j\}\]
for $1\leq i\leq n$. By construction $\lambda_i (S) \leq \lambda_j(S)$ for $i\leq j$.
\end{defn}

\begin{defn}\label{def:adelinhommin}
The \emph{inhomogeneous minimum} of the adelic convex body $S$ is
\[\mu(S)\coloneqq\inf\Bigl\{\mu>0 \Bigm| K_\bbA^n \subseteq \bigcup_{\zeta\in K^n} \bigl(\mu S+\zeta\bigr)\Bigr\}\,.\]
By construction $\mu(S) = \widehat{\mu}(\rho(S_\infty),\rho(\iota(\frakM)))$,
where 
\[\widehat{\mu}(T,\Lambda)\coloneqq\inf\Bigl\{\mu>0 \Bigm| \bbR^m\subseteq \bigcup_{\zeta\in \Lambda} \bigl(\mu T+\zeta\bigr) \Bigr\}\]
is the classical inhomogeneous minimum of the convex body $T\subset\bbR^m$ with respect to the lattice $\Lambda\subset\bbR^m$, cf.\ \cite[\textsection\,5]{MR1242995}.
Here $\frakM=\bigcap_{v\nmid\infty} \bigl(S_v \cap K^n\bigr)$.
\end{defn}

\section{Adelic Polarity}

In order to define our notion of adelic polarity we first recall some
background from Algebraic Number Theory.
It is well-known~\cite[Ch.\,I,(2.8)]{neukirch}, that 
\[T(x,y)\coloneqq\Tr_{K/\bbQ}(xy)\]
is a non-degenerate symmetric $\bbQ$-bilinear form on $K$.
Here $\Tr_{K/\bbQ}$ denotes the field trace.
This allows to define
\begin{equation}
\compmod\coloneqq\Set{x\in K | \Tr_{K/\bbQ}(xy)\in\bbZ\ \forall\,y\in\calO}\,,
\label{eq:defcompmod}
\end{equation}
the \emph{complementary module}, cf.~\cite[Ch.\,III,\,\textsection\,2]{neukirch}.
This is a fraction ideal in $K$, its inverse is the \emph{different} $\frakd$.
On $K^n$ we get a bilinear form given by 
\[T_n( x,y)\coloneqq\sum_{i=1}^n \Tr_{K/\bbQ}(x_i y_i)\,.\]

\bigskip

By \cite[V\,§\,2,\,Thms.\,2\,\&\,3]{weilbasicnumber} for any fractional ideal $\frakm$ 
there is a map $a\colon M(K)\rightarrow\bbZ$, such that $\frakm$ can be written as
\begin{equation}\label{eq:idealasintersection}
\frakm=\bigcap_{v\nmid\infty}(K\cap \frakp_v^{a(v)})\,,
\end{equation}
where almost all $a(v)=0$ and $\frakp_v$ is the unique maximal ideal in $\calO_v$.
More concretely, we get the following special case.

\begin{lem}\label{lem:caloastalsschnitt}
Let $v\nmid\infty$ and define as in the global case
\[\compmod_v\coloneqq\Set{x\in K_v | \Tr_{K_v/\bbQ_v}(xy)\in\bbZ_v\ \forall\,y\in\calO_v}\,.\]
Then $\compmod=\bigcap_{v\nmid\infty}(\compmod_v\cap K)$.
For all but finitely many $v\nmid\infty$ we have $\compmod_v=\calO_v$.
\end{lem}

\begin{proof}
By their definitions (cf.~\cite[p.\,377\,($\star$)]{knappadvancedalgebra}) we have
\[ \compmod_v \cap K = \compmod_{(v)} \coloneqq 
\Set{ \tfrac{a}{b} | a \in \compmod, b \in \calO \setminus (v)}\supseteq\compmod,\]
where $\compmod_{(v)}$ is the localisation of $\compmod$ at the ideal $(v)$ corresponding to $v$.

For the converse inclusion we follow an idea by J.\,Jahnel\footnote{personal communication}.
Let $M\coloneqq \bigcap_{v\nmid\infty} \compmod_{(v)} $, $x\in M$ and
consider the “ideal of denominators”
\[I\coloneqq\Set{ b\in \calO | bx \in M}\,.\]
Since $x\in K\cap\compmod_v=\compmod_{(v)}$, we have $I\not\subset(v)$, 
for the ideal in $K$ corresponding to $v$.
Since this holds for all $v$, we have $I=\calO$. Therefore $x\in\compmod$.

The final statement follows from \cite[Lemma~6.48]{knappadvancedalgebra},
since only finitely many primes are ramified in $K$.
\end{proof}

We extend the construction from \eqref{eq:defcompmod} in a natural way to the rank-$n$ case with the form $T_n$.

\begin{lem}\label{lem:dimensionn}
Let $A\in\GL_n(K)$ and $A_v\in\GL_n(K_v)$ for any finite $v$.
Then
\[\prescript{\star}{}{(A\calO^n)} = A^{-t} (\compmod)^n\quad\text{and}\quad
\prescript{\star}{}{(A_v^{\phantom{x}}\calO_v^n)} = A_v^{-t} (\compmod_v)^n\,.\]
\end{lem}

\begin{proof}
Notice, that
\[\prescript{\star}{}{(\calO^n)}\coloneqq\set{x \in K^n | T_n( x,y)\in\bbZ\ \forall\, y \in \calO^n} \supseteq (\compmod)^n\,.\]
Suppose they are not the same, i.e.\ $\exists\,a\in\prescript{\star}{}{(\calO^n)}\setminus (\compmod)^n$.
Then for some $i$: $a_i\not\in\compmod$, so there is some $b_i\in\calO$, 
such that $\Tr_{K/\bbQ}(a_ib_i)\not\in\bbZ$ by definition of $\compmod$.
But then $T_n(a,(0,\ldots,0,b_i,0,\ldots,0))\not\in\bbZ$ giving a contradiction.

Now let $(a_{ij})_{ij}=A\in\GL_n(K)$, $x,y\in K^n$.
Then 
\begin{align*}
T_n(x, Ay)&=\sum_i \Tr_{K/\bbQ}(x_i (Ay)_i)
=\sum_i \Tr_{K/\bbQ}\bigl(x_i \bigl({\textstyle\sum_j} a_{ij}y_j\bigr)\bigr)\\
&=\sum_i \sum_j\Tr_{K/\bbQ}\bigl(x_i ( a_{ij}y_j)\bigr)
=\sum_j \sum_i\Tr_{K/\bbQ}\bigl((a_{ij}x_i) y_j\bigr)\\
&=\sum_j \Tr_{K/\bbQ}((A^t x)_j y_j)=T_n( A^t x, y)\,.
\end{align*}
The second statement is obvious, as the above argument
works for $x,y\in K_v^n$ and $A_v\in\GL_n(K_v)$ verbatim using $\Tr_{K_v/\bbQ_v}$.
\end{proof}

\medskip

On the other hand, we can define a scalar product on $\bbR^d=\bbR^{r+2s}$ as
\begin{equation}
(x,y)=\sum_{i=1}^r x_i y_i + 2 \sum_{i=r+1}^{2s} x_iy_i\,,\label{eq:skalprodaufRn}
\end{equation}
cf.~\cite[Ch.\,I,(5.1)]{neukirch}. This gives the scalar product
\begin{equation}\label{eq:skalprodaufKinfinity}
(x,y)\coloneqq(\rho(x),\rho(y))=\sum_{v\text{ real}} x_v y_v + \sum_{v\text{ complex}} (x_v\overline{y}_v +\overline{x}_v y_v)
\end{equation}
on $K_\infty$, cf.~\cite[p.\,222]{neukirch}.

\begin{lem}\label{lem:bilformsareequal}
For all $x,y\in K$
\[ 
\Tr_{K/\bbQ}(x y) = \left( \rho(\iota(x)),\rho(\overline{\iota}(y))\right)\,.
\]
\end{lem}

\begin{proof}
Let $x,y\in K$, then
\begin{align*}
&\phantom{=}\left(\rho(\iota(x)),\rho(\overline{\iota}(y))\right)\\
&\phantom{=}=\sum_{j=1}^r \sigma_j(x) \sigma_j(y) 
+ \sum_{j=1}^s 2\bigl(\Rea(\sigma_{r+j}(x))\Rea(\overline{\sigma}_{r+j}(y))+ \Ima(\sigma_{r+j}(x))\Ima(\overline{\sigma}_{r+j}(y))\bigr)\\
&\phantom{=}=\sum_{j=1}^r \sigma_j(x) \sigma_j(y) 
+ 2\sum_{j=1}^s \bigl(\Rea(\sigma_{r+j}(x))\Rea(\sigma_{r+j}(y))- \Ima(\sigma_{r+j}(x))\Ima(\sigma_{r+j}(y))\bigr)\,.
\end{align*}

By~\cite[Ch.\,I,(2.6)\,(ii)]{neukirch}, $\Tr_{K/\bbQ}(x)=\sum_\sigma \sigma(x)$, 
where the sum is over all embeddings $\sigma\colon K\hookrightarrow \overline{\bbQ}$.
As all complex embeddings appear in conjugated pairs
\begin{align*}
\Tr_{K/\bbQ}(xy) &= 
\sum_{j=1}^r \sigma_j(xy) + \sum_{j=1}^s \sigma_{r+j}(xy) + \sum_{j=1}^s \overline{\sigma}_{r+j}(xy) \\
&= \sum_{j=1}^r \sigma_{j}(x)\sigma_{j}(y) + 2\sum_{j=1}^s \Rea(\sigma_{r+j}(x)\sigma_{r+j}(y))\,.
\end{align*}
The statement follows from $\Rea(ab)=\Rea(a)\Rea(b)-\Ima(a)\Ima(b)$.
\end{proof}

\begin{cor}\label{cor:polariscompl}
For any algebraic number field $K$ with ring of integers $\calO$ and embeddings $\rho$ and $\iota$ as above, we have
\[\rho(\iota(\calO))^\star=\rho(\overline{\iota}(\compmod))\,,\]
where $(\,\cdot\,)^\star$ is the polar with respect to the form in (\ref{eq:skalprodaufRn}).
\end{cor}

The scalar product $(\,\cdot\,,\,\cdot\,)$ on $\bbR^{nd}$ is also defined 
as the sum of the components of each copy of $\bbR^d$.
Notice that we get the standard scalar product at the real places and 
the real scalar product multiplied by $2$ at the complex places.
By direct consequence of Lemma~\ref{lem:dimensionn} and Corollary~\ref{cor:polariscompl}
and again \cite[V\,§\,2,\,Thm.\,2]{weilbasicnumber}, cf.\ (\ref{eq:idealasintersection}),
this leads to the following generalisation.

\begin{cor}\label{cor:scalarproductindimnd}
For any algebraic number field $K$ with ring of integers $\calO$ and embeddings $\rho^n$ and $\iota^n$ as above, we have
\begin{align*}
\rho^n(\iota^n(A^{-1}\calO^n))^\star&=\rho^n(\overline{\iota}^n(A^t(\compmod)^n))\\
\shortintertext{and}
\rho^n\Bigl(\iota^n\Bigl(\bigcap_{v\nmid\infty}(A_v^{-1}\calO_v^n\cap K^n)\Bigr)\Bigr)^\star
&=\rho^n\Bigl(\overline{\iota}^n\Bigl(\bigcap_{v\nmid\infty}(A_v^t(\compmod_v)^n\cap K^n)\Bigr)\Bigr)
\end{align*}
for $A\in\GL_n(K)$, $A_v\in\GL_n(K_v)$ for all $n\in\bbN$.
\end{cor}

\begin{rem}
Consider a finite number of $0$-symmetric convex bodies $S_i\subset\bbR^{m_i}$. 
Then, using the classical notion of polarity,
\begin{equation}\label{eq:prodofkomplcontainscomplofprod}
\bigl(\prod_i S_i\bigr)^\star \subseteq \prod_i S_i^\star\,.
\end{equation}
Indeed, let $x\in\bigl(\prod_i S_i\bigr)^\star$, then $\left<x,y\right>\leq 1$ 
for all $y\in \prod_i S_i$. So especially for any $i$ we have
$\left<x,(0,\ldots,0,y_i,0,\ldots,0)\right>\leq 1$ for all $y_i\in S_i$.
But that implies $\left<x_i,y_i\right>\leq 1$ for all $i$, 
which defines the right-hand side of (\ref{eq:prodofkomplcontainscomplofprod}).

For the scalar product $(\,\cdot\,,\,\cdot\,)$ instead of $\left\langle \,\cdot\,,\,\cdot\,\right\rangle$
we get 
$\left<x,(0,\ldots,0,y_i,0,\ldots,0)\right>\leq \tfrac{1}{2}$ and
$\left<x_i,y_i\right>\leq \tfrac{1}{2}$ for the complex places ($x_i,y_i\in\bbC$), 
so (\ref{eq:prodofkomplcontainscomplofprod}) holds as well.
\end{rem}

Due to Corollary~\ref{cor:scalarproductindimnd}, we are now in the situation to define our notion of adelic polarity.

\begin{defn}\label{def:adelicpolarbody}
Let $S = \prod_{v\nmid\infty} A_v^{-1}\calO_v^n \times \prod_{v\mid\infty} S_v$ be an adelic convex body.
The \emph{polar adelic body} of $S$ is
\[S^\star\coloneqq\prod_{v\nmid\infty} A_v^{t} (\compmod_v)^n \times\prod_{v\mid\infty} S_v^\star\,,\]
where $S_v^\star$ is the polar body of $S_v$ with respect to the restriction of (\ref{eq:skalprodaufRn}).
Since $\calO_v=\compmod_v$ for almost all $v\nmid\infty$ by Lemma~\ref{lem:caloastalsschnitt}, 
$S^\star$ is again an adelic convex body.
\end{defn}

\section{Adelic Transference Theorems}\label{sec:mainresults}

We now apply the results of the previous section, 
especially Corollary~\ref{cor:scalarproductindimnd}, 
to prove the main results of this paper.

\begin{proof}[Proof of Theorem~\ref{thm:adelicpolarupper}]
Let
\[\frakM=\bigcap_{v\nmid\infty} \bigl(A_v^{-1} \calO_v^n \cap K^n\bigr)
\quad\text{and}\quad
\frakM^\star=\bigcap_{v\nmid\infty} \bigl(A_v^{t} (\compmod_v)^n \cap K^n\bigr)\,.\]
By \cite[Lemma]{thunderremarksonadelic} $\rho(\iota(\frakM))$ 
and $\rho(\overline{\iota}(\frakM^\star))$ are lattices of full rank in $\bbR^{nd}$.
By Corollary~\ref{cor:scalarproductindimnd}, they are polar to each other.

Denote by $S_\infty$ and $S_\infty^\star$ the infinite parts of
$S$ and $S^\star$ respectively. 
By (\ref{eq:prodofkomplcontainscomplofprod}) we have
\begin{equation}\label{eq:dualindual}
(\rho(S_\infty))^\star\subset \rho(S_\infty^\star)\,.
\end{equation}

Denote by $\lambda_\ell(S)$ and $\lambda_\ell(S^\star)$ the adelic successive minima 
of $S$ and $S^\star$ respectively
and by $\widehat{\lambda}_i(T,\Lambda)$ the classical successive minima 
of the convex body $T$ and the lattice $\Lambda$ in $\bbR^{nd}$.
Then, by \cite[p.\,256]{thunderremarksonadelic}, for $\ell=1,\ldots,n$
\begin{align*}
\lambda_\ell(S) &\leq 
\widehat{\lambda}_{(\ell-1)d+1}\bigl(\rho(S_\infty),\rho(\iota(\frakM))\bigr)\\
\shortintertext{and}
\lambda_\ell(S^\star) &\leq \widehat{\lambda}_{(\ell-1)d+1}
\bigl(\rho(S_\infty^\star),\rho(\overline{\iota}(\frakM^\star))\bigr)
\leq \widehat{\lambda}_{(\ell-1)d+1}
\bigl(\rho(S_\infty)^\star,\rho(\overline{\iota}(\frakM^\star))\bigr)\,,
\end{align*}
where the last inequality follows from (\ref{eq:dualindual}).

Finally, applying (\ref{eq:classicalinequality}), we conclude
\begin{align*}
\lambda_\ell(S)\lambda_{n-\ell+1}(S^\star)&\leq 
\widehat{\lambda}_{(\ell-1)d+1}(\rho(S_\infty),\rho(\iota(\frakM)))
\widehat{\lambda}_{((n-\ell+1)-1)d+1}(\rho(S_\infty)^\star,\rho(\overline{\iota}(\frakM^\star)))\\
&\leq \widehat{\lambda}_{(\ell-1)d+1}(\rho(S_\infty),\rho(\iota(\frakM)))
\widehat{\lambda}_{(n-\ell)d+d}(\rho(S_\infty)^\star,\rho(\overline{\iota}(\frakM^\star)))\\
&\leq (nd)^{3/2} \,.\qedhere
\end{align*}
\end{proof}

\begin{cor}
Let $K$, $S$, $S^\star$ and $\lambda_1(S)$ be as in Theorem~\ref{thm:adelicpolarupper}
and let $\mu(S^\star)$ be the inhomogeneous minimum of $S^\star$. Then
\[\lambda_1(S) \cdot \mu(S^\star) \leq C\, nd(1+\log nd)\,,\]
where $C$ is a universal constant.
\end{cor}

\begin{proof}
As in the proof of Theorem~\ref{thm:adelicpolarupper},
we have $\lambda_1(S) = \widehat{\lambda}_1(\rho(\iota(\frakM)),\rho(S_\infty))$
and by (\ref{eq:dualindual}) we get
\[\widehat{\mu}(\rho(S_\infty^\star),\Lambda) \leq \widehat{\mu}(\rho(S_\infty)^\star,\Lambda)\]
for any lattice $\Lambda\subset\bbR^{nd}$.
Therefore
\[\lambda_1(S) \cdot \mu(S^\star) \leq \widehat{\lambda}_1(\rho(S_\infty),\rho(\iota(\frakM)))\cdot
 \widehat{\mu}(\rho(S_\infty)^\star,\rho(\overline{\iota}(\frakM^\star))) \leq C\, nd(1+\log nd)\,,
\]
by \cite[Corollary~1]{banaszcykinequalities2} with some universal constant $C$.
\end{proof}

\begin{proof}[Proof of Theorem~\ref{thm:adelicpolarlower}]
We use the standard bilinear form on $K^n$:
\[b(x,y)=\sum_{i=1}^n x_i \overline{y}_i\,,\]
where $\overline{\,\cdot\,}$ is complex conjugation if $K$ is a CM-field 
and the identity for $K$ totally real.
Let $u_1,\ldots,u_n$ and $v_1,\ldots,v_n$ be $K$-bases of $K^n$ such that
$u_i\in \lambda_i(S) S$ and $v_j\in \lambda_j(S^\star) S^\star$ for all $i,j$.
Notice that for $u_i\in\calO^n$ and $v_j\in(\compmod)^n$, we have
$b(A^{-1}u_j,\overline{A^{t}v_j})=b(u_j,\overline{v}_j)\in\compmod$,
using that $\compmod$ is a fractional ideal in $K$.
By definition of $\compmod$ and the different $\frakd$, we have
$\abs{x}\leq\abs{\frakd}^{-1}$ for $x\in\compmod$, \cite[III,\,2.1]{neukirch}.
This holds for any finite place $v$ as well.

Since $b$ is non-degenerate, there are
$i\in\Set{1,\ldots,\ell}$ and $j\in\Set{1,\ldots,n-\ell+1}$ such that $b(u_i,\overline{v}_j)\neq 0$.
Then by the product formula in (\ref{eq:productformulaetc})
\begin{align*}
1&=\prod_v \abs{b(u_i,\overline{v}_j)}_v^{d_v}\cdot
\left(\frac{\lambda_i(S)\lambda_j(S^\star)}{\lambda_i(S)\lambda_j(S^\star)}\right)^d\\
 &=\prod_{v\nmid\infty} \bigl|b(u_i,\overline{v}_j)\bigr|_v^{d_v}
 \cdot \left(\lambda_i(S)\lambda_j(S^\star)\right)^d
 \cdot\prod_{v\mid\infty} 
   \bigl|b(\tfrac{1}{\lambda_i(S)}u_i,\tfrac{1}{\lambda_j(S^\star)}\overline{v}_j)\bigr|_v^{d_v}\,.
\end{align*}
Now for any finite $v$ we have $b(u_i,\overline{v}_j)\in\compmod_v$, therefore $\bigl|b(u_i,\overline{v}_j)\bigr|_v^{d_v}\leq \abs{\frakd_v}^{-d_v}$,
where $\frakd_v$ denotes the local different. 
Finally $\prod_{v\nmid\infty}\abs{\frakd_v}^{-d_v}=\abs{\Delta_K}$, cf.~\cite[Ch.\,VI,\,§\,8]{knappadvancedalgebra}.

To conclude the proof, we consider the factors at the infinite places.
By assumption they are either all real or all complex. Fix some $v\mid\infty$.
Let $x\coloneqq\tfrac{1}{\lambda_i(S)}u_i$ and $y\coloneqq\tfrac{1}{\lambda_j(S^\star)}v_j$.
If $K$ is totally real, i.e.\ $v$ is real, we have
\[\bigl|b(x,\overline{y})\bigr|_v^{d_v}=\bigl|b(x,y)\bigr|_v^1
=\bigl|\sigma_v\bigl(\tsum_i x_i y_i \bigr)\bigr|
=\bigl|\tsum_i \sigma_v(x_i) \sigma_v(y_i) \bigr|
\leq 1\,,\]
by definition of $S_v^\star$.

If $K$ is a CM-field, i.e.\ $v$ is complex, we get
\[\begin{multlined}
\bigl|b(x,\overline{y})\bigr|_v^{d_v}
=\bigl|\sigma_v\bigl(\tsum_i x_i \overline{y}_i \bigr)\bigr|^2
=\bigl|\tsum_i \sigma_v(x_i) \overline{\sigma_v(y_i)} \bigr|^2\\
\leq \bigl(\bigl|\Rea(\tsum_i \sigma_v(x_i) \overline{\sigma_v(y_i)})\bigr| 
+ \bigl|\im\Ima(\tsum_i \sigma_v(x_i) \overline{\sigma_v(y_i)})\bigr|\bigr)^2
\leq \left(\left|\tfrac{1}{2}\right|+1\left|\tfrac{1}{2}\right|\right)^2=1\,,\end{multlined}\]
by definition of $S_v^\star$, since $\im\Ima(x)=\im\Rea(\im x)$ for all $x\in\bbC$
and from $(\sigma_v(x_i))_i\in S_v$ we get  $\im(\sigma_v(x_i))_i\in S_v$ by our additional requirement.

The conclusion follows from the monotonicity of the minima.
\end{proof}

\begin{exmp}\label{exmp:bbQsqrt2}
Let $n=1$ and $K=\bbQ[\sqrt{2}]$, then $\calO=\bbZ[\sqrt{2}]=\bbZ+\sqrt{2}\bbZ$ 
and the field discriminant is $\abs{\Delta_K}=8$.
Consider $x=a+b\sqrt{2}\in\bbQ[\sqrt{2}]$ and $y=c+d\sqrt{2}\in\bbZ[\sqrt{2}]$.
Then
\[xy=(a+b\sqrt{2})(c+d\sqrt{2})=ac+2bd+(ad+bc)\sqrt{2}\,.\]
Therefore
\[\Tr(xy)=\Tr\begin{pmatrix}ac+2bd&2ad+2bc\\ ad+bc&ac+2bd\end{pmatrix}=2ac+4bd\]
and this is an integer if $a\in\tfrac{1}{2}\bbZ$ and $b\in\tfrac{1}{4}\bbZ$.
Therefore $\compmod=\tfrac{1}{2}\bbZ+\tfrac{\sqrt{2}}{4}\bbZ$.

Now $\rho(\iota(\calO)),\rho(\iota(\compmod))\subset\bbR^2$ are lattices of rank $2$,
more precisely
\[\rho(\iota(\calO))=\begin{pmatrix}1&\sqrt{2}\\1&-\sqrt{2}\end{pmatrix}\bbZ^2\,,\qquad
\rho(\iota(\compmod))=\begin{pmatrix}\tfrac{1}{2}&\tfrac{\sqrt{2}}{4}\\\tfrac{1}{2}&-\tfrac{\sqrt{2}}{4}\end{pmatrix}\bbZ^2\,,\]
and we see that $\rho(\iota(\calO))^\star=\rho(\iota(\compmod))$.
This follows easily from the fact, that the matrices are the inverse transpose of one another.

Taking the $1$-dimensional unit ball $[-1,1]$ at both infinite places for the convex bodies,
we see that
\[S=\prod_{v\nmid\infty}\calO_v\times\prod_{v\mid\infty}[-1,1]\quad\text{and}\quad
S^\star=\prod_{v\nmid\infty}\compmod_v\times\prod_{v\mid\infty}[-1,1]
\]
are polar. Obviously $\lambda_1(S)\leq 1$ and since $\tfrac{\sqrt{2}}{4}<\tfrac{1}{2}$,
we have $\lambda_1(S^\star)\leq\tfrac{\sqrt{2}}{4}$.
This gives equality for the lower bound in Theorem~\ref{thm:adelicpolarlower}.
\end{exmp}

\medskip
\noindent{\itshape Acknowledgment.} I would like to thank Martin Henk, 
Florian Heß, Matthias Henze, Jörg Jahnel and Kristin Stroth for helpful comments 
and discussions on the subject.

\bibliographystyle{amsplain}
\bibliography{polarbody-paper}

\providecommand{\bysame}{\leavevmode\hbox to3em{\hrulefill}\thinspace}
\providecommand{\MR}{\relax\ifhmode\unskip\space\fi MR }
\providecommand{\MRhref}[2]{%
  \href{http://www.ams.org/mathscinet-getitem?mr=#1}{#2}
}
\providecommand{\href}[2]{#2}
\begin{thebibliography}{10}

\bibitem{banaszcyknewbounds}
Wojciech Banaszczyk, \emph{New bounds in some transference theorems in the
  geometry of numbers}, Math. Ann. \textbf{296} (1993), no.~4, 625--635.

\bibitem{banaszcykinequalities2}
\bysame, \emph{Inequalities for convex bodies and polar reciprocal lattices in
  {$\mathbf R^n$}. {II}. {A}pplication of {$K$}-convexity}, Discrete Comput.
  Geom. \textbf{16} (1996), no.~3, 305--311.

\bibitem{bombierivaalersiegelslemma}
Enrico Bombieri and Jeffrey~D. Vaaler, \emph{On {S}iegel's lemma}, Invent.
  Math. \textbf{73} (1983), no.~1, 11--32.

\bibitem{fukshanskysiegelslemma}
Lenny Fukshansky, \emph{Siegel's lemma with additional conditions}, J. Number
  Theory \textbf{120} (2006), no.~1, 13--25.

\bibitem{fukshanskyalgebraicpoints}
\bysame, \emph{Algebraic points of small height missing a union of varieties},
  J. Number Theory \textbf{130} (2010), no.~10, 2099--2118.

\bibitem{gaudron}
{\'E}ric Gaudron, \emph{G\'eom\'etrie des nombres ad\'elique et lemmes de
  {S}iegel g\'en\'eralis\'es}, Manuscripta Math. \textbf{130} (2009), no.~2,
  159--182.

\bibitem{MR1242995}
Peter~M. Gruber, \emph{Geometry of numbers}, Handbook of convex geometry,
  {V}ol.\ {A}, {B}, North-Holland, Amsterdam, 1993, pp.~739--763.

\bibitem{knappadvancedalgebra}
Anthony~W. Knapp, \emph{Advanced algebra}, Cornerstones, Birkh\"auser Boston
  Inc., Boston, MA, 2007.

\bibitem{neukirch}
J{\"u}rgen Neukirch, \emph{{A}lgebraische {Z}ahlentheorie}, Springer-Verlag,
  Berlin etc., 1992.

\bibitem{rothlisberger}
Mark~Peter Rothlisberger, \emph{An analogue of the {K}orkin-{Z}olotarev lattice
  reduction for vector spaces over number fields}, Ph.D. thesis, The University
  of Texas at Austin, 2010.

\bibitem{roythunderabsolutesiegel}
Damien Roy and Jeffrey~Lin Thunder, \emph{An absolute {S}iegel's lemma}, J.
  Reine Angew. Math. \textbf{476} (1996), 1--26.

\bibitem{schmidt.lnm1467}
Wolfgang~M. Schmidt, \emph{Diophantine approximations and {D}iophantine
  equations}, Lecture Notes in Mathematics, vol. 1467, Springer-Verlag, Berlin,
  1991.

\bibitem{thunderremarksonadelic}
Jeffrey~Lin Thunder, \emph{Remarks on adelic geometry of numbers}, Number
  theory for the millennium, {III} ({U}rbana, {IL}, 2000), A K Peters, Natick,
  MA, 2002, pp.~253--259.

\bibitem{weilbasicnumber}
Andr{\'e} Weil, \emph{Basic {N}umber {T}heory}, Classics in Mathematics,
  Springer-Verlag, Berlin, 1995, Reprint of the second (1973) edition.

\end{thebibliography}
\end{document}